\documentclass[a4paper,12pt]{amsart}
\usepackage[margin=1.65in,footskip=0.50in]{geometry}
\usepackage{amsfonts}
\usepackage{mathrsfs}
\usepackage{amsmath,amssymb,latexsym,amsfonts,amscd}
\usepackage[usenames,dvipsnames]{color}
\usepackage{multicol}
\usepackage{verbatim}
\usepackage[all]{xy}

\newcommand\PPP{{\mathbb{P}}}

\newcommand\CC{{\mathbb{C}}}

\newtheorem{thm}{Theorem}[section]
\newtheorem{lem}[thm]{Lemma}

\newtheorem{prop}[thm]{Proposition}
\newtheorem{definition}[thm]{Definition}

\newtheorem{corollary}[thm]{Corollary}

\theoremstyle{definition}
\newtheorem{defn}[thm]{Definition}

\newtheorem{rem}[thm]{Remark}
\newtheorem{notation}[thm]{Notation}
\newtheorem{convention}[thm]{Convention}
\theoremstyle{remark}

\numberwithin{equation}{thm}

\title{On Higher Dimensional Extremal varieties of General Type}
\author{Purnaprajna Bangere}
\address{Mathematics Department, 
University of Kansas,
Lawrence, 
USA}

\email{purna@ku.edu}

\author{Jungkai A. Chen}

\address{Department of Mathematics, 
National Taiwan University,  Taipei, Taiwan}
\address{National Center for Theoretical Sciences, Taipei, Taiwan}
\email{jkchen@ntu.edu.tw}

\author{Francisco J. Gallego}

\address{Departamento de \'Algebra,
Geometr\'ia y Topolog\'ia and
Instituto de Matem\'atica Interdisciplinar, Universidad Complutense de Madrid, Madrid,
Spain}

\email{gallego@mat.ucm.es}



\thanks{The first author was partially supported by GRF grant of the University of Kansas.
The first author is also grateful for the hospitality of the Departamento de \'Algebra, Geometr\'ia y Topolog\'ia of the Universidad Complutense de Madrid
and for the support of grant MTM2015-65968-P during his visit.
The second author was partially supported by the National Center for Theoretical Sciences
and the Ministry of Science and Technology  of Taiwan.
The third author was partially
supported by grant MTM2015-65968-P and by UCM research group 910772.}

\begin{document}

\begin{abstract} Relations among fundamental
invariants play an important role in algebraic geometry.
 It is known
 that an $n$-dimensional  variety
 of general type with nef canonical divisor and canonical singularities,
 whose image $Y$ under the
 canonical map is of maximal dimension, satisfies $K_X^n \ge 2 (p_g-n)$.
 We investigate the very interesting extremal situation $K_X^n=2(p_g-n)$,
which appears in a number of geometric situations.
Since these extremal varieties are natural higher
 dimensional analogues of Horikawa surfaces,
 we name them Horikawa varieties. These
 varieties have been previously dealt with in
 the works of Fujita \cite{Fujita} and Kobayashi \cite{Kobayashi}.
{We carry out further studies of
{Horikawa} varieties,
 proving new results on various geometric and topological issues concerning them.
 In particular, we prove that {the}
 geometric genus of those Horikawa varieties whose
 image under the canonical map is singular is bounded.
 {We give an analogous result
 for polarized hyperelliptic subcanonical varieties, in particular,
 for polarized Calabi-Yau  and Fano varieties.}
 The pleasing numerology that emerges puts Horikawa's
 result on surfaces in a broader perspective.
 We obtain a structure theorem for Horikawa varieties
 and explore their pluriregularity. We use this to prove
 optimal results on projective normality of pluricanonical
 linear systems. We study the fundamental groups of Horikawa varieties, showing that
 they are simply connected, even if $Y$ is singular.
 We also prove results on deformations of Horikawa varieties,
 {whose implications on the moduli space make
 them the higher dimensional analogue of curves of genus $2$.}}
 \end{abstract}

\maketitle
\section{Introduction}

      For a minimal surface of general type it is a well known result of Noether that its canonical divisor $K_S$ satisfies
      $K_S^2 \ge 2p_g-4$. The surfaces for which $K_S^2 = 2p_g-4$
      have been dealt by Horikawa in his well known work  \cite{Horikawa}.
      Horikawa shows that, for surfaces of general type on the Noether line, $K_S$ is indeed base point free and the complete linear system $|K_S|$
      maps $S$ as a canonical double cover of a surface of minimal degree. These and, more generally, canonical double covers of rational surfaces have a
      ubiquitous presence in the geometry of algebraic surfaces.

  \medskip

   In this article, we prove higher dimensional analogues
   of the results of \linebreak Horikawa.  Some very interesting work has
   been previously done by Fujita and Kobayashi.
  {Fujita and Kobayashi proved
   (see \cite{otherFujita} and
   \cite[Theorem 2.4]{Kobayashi}) that,}
   if $X$ is a variety of general type {with nef canonical divisor and at worst
   canonical  singularities},
   whose image under the
 canonical map
   has maximum dimension, then $K_{X}^n \ge 2(p_g(X)-n)$.
{In
 this article we study, \emph{Horikawa varieties}, i.e.,
 those varieties for which the equality holds:}

 \begin{definition}\label{defi.Horikawa}
 {Let $X$ be a variety of general type of dimension $n$, $n \geq 2$  {with at worst canonical singularities and whose canonical divisor $K_X$ is nef}.
We say that $X$ is a Horikawa variety if
 \begin{enumerate}
  \item the image of the canonical map of $X$ has dimension $n$; and
  \item $K_{X}^n = 2(p_g(X)-n)$.
 \end{enumerate}}
\end{definition}

    {If $X$ is a Horikawa variety of arbitrary dimension,
    then $K_X$ is a Cartier divisor (and, hence, it has Gorenstein singularities) and the linear system $|K_X|$ is base--point--free
    (see  {\cite[Propositions 2.2, 2.5]{Kobayashi}}).}
   {In addition, the canonical morphism of $X$
   is a generically finite
   morphism of degree $2$ onto
    a variety of minimal degree
    {(see \cite[Proposition 2.5]{Kobayashi}}).
    In fact this property characterizes those
    varieties of general type  {with nef canonical divisor and canonical singularities}
    which are Horikawa varieties.}

\medskip

     {In Section 2} we study various geometric and
     {cohomological} aspects of these extremal
     varieties.
     We show that Horikawa varieties are regular
     (see Theorem~\ref{regular}).
    {We also study the deformations of the canonical morphism
    of Horikawa varieties.  We prove that if $X$ is a
    Horikawa variety of any dimension, then the deformation of its,
    degree $2$, canonical morphism are again
 canonical of degree $2$ whose image is a variety of minimal degree.}
 Some of the cases have an overlap with results in \cite{Fujita},
   but, even in those instances, the methods of
    this article are very different and more transparent. There are implications of our results to the moduli space of varieties of general type.
    In this article, we also prove a structure theorem for Horikawa varieties
    {(see Theorem~\ref{regularity}):}
    If $X$ is a Horikawa variety of dimension $n$ and $Y$
    is smooth, then it is pluriregular
    {and, in most cases,} it possesses a
    fibration, with fibers $F$  of general type and $p_g(F)=n$.

    \medskip

    In Section 2, we prove an interesting result (see
    Theorem~\ref{theorem.p_g.bounded})
    which shows that geometric genus of a Horikawa variety of dimension $n$ is bounded if the image under the canonical morphism
    is a singular variety $Y$.
     Precisely, we
     show that  $p_g(X)\leq n+4$. This gives raise
    to a beautiful numerology; indeed, this bound
    generalizes Horikawa's bound in
    the case of surfaces, which is $p_g \leq 6$. In our proof,
    we consider the pullback $X_2$ of the intersection $Y_2$ of $n-2$
    hyperplane sections passing
    through a given point of $Y$.
    The surface $X_2$ maps onto $Y_2$ by a subcanonical linear
    series.
    Even though $X$ has at worst canonical
    singularities, the singularities might a priori
    get worse and worse, since we are taking special hyperplane sections to obtain
    $X_2$. We are able to show that $X_2$ has, nevertheless,
    only canonical
    singularities. This allows us to use in a crucial way the subtle
    arguments developed in the proof of Theorem~\ref{regular} to bound
    the dimension of the projective space containing $Y_2$.

     \medskip
    {In Section 3 we study
    topological aspects of Horikawa varieties.
    Precisely, we show that
   Horikawa varieties  are simply connected (see
   Theorem~\ref{topology}).}
  {Although Fujita proved previously in \cite{Fujita} the simple connectedness of some families of Horikawa varieties, our} proof is very transparent and covers all the cases.  Our methods are quite different from the methods in \cite{Fujita}. We crucially use results of Nori in \cite{Nori}.
    {Even though simple connectedness
   implies regularity, the proof of Theorem~\ref{regular}
   is independent, and necessary because
   of the reasons explained  in the
   paragraph above.}

    \medskip

    The classification of
    Horikawa varieties involves canonical
    covers of varieties of minimal degree.
    The canonical covers of varieties of minimal degree have a significant presence in the geometry of algebraic surfaces and higher dimensional varieties
    of general type and they occur in a variety of contexts such as determination of very ampleness of linear series, ring generation, deformation theory and
    construction of varieties with given invariants (see \cite{Triplecovers}, \cite{Calabi-Yau},
    {\cite{Inventiones}},
\cite{TCMDF}, \cite{DCMDF},  \cite{Horikawa}).
    One of the most natural contexts where the
    canonical covers of minimal degree varieties occur is at the boundary of the geography of surfaces of general type.
    The results in this article and those in \cite{Kobayashi}
    and \cite{Fujita}, show that this is true for all
    higher dimensional varieties of general type as well on what
    can be called the ``Noether faces''.
\medskip

    It is indeed compelling to note the beauty of the
    analogy with lower dimensional results,
    even though methods are different. The analogies are striking,
    despite the existence
    of much worse singularities.
{In view of Theorem~\ref{deformation}, smooth Horikawa varieties can
   be considered as analogues of curves of genus $2$,
   from the point of view of deformations and moduli.
   Moreover,} the implications of our results on classification,
    regularity and deformations of Horikawa varieties
   yield interesting consequences for the moduli space
     of varieties  of general type (see Corollary~\ref{cor.moduli}).

 \medskip

The question of projective normality of pluricanonical systems of a surface of general type has a long history dating back to Kodaira and Bombieri.
In this article we also show that if $X$ is a Horikawa variety of dimension $n$ with an optimal condition on its geometric genus, then $|mK_X|$ embeds $X$ as a projectively normal
variety if and only if  $m\geq n+1$. The standard methods involving Castelnuovo-Mumford regularity and vanishing theorems do not yield this result.
We do in fact show a general statement that follows from these standard methods but, to get the optimal statements, one has to use in an essential way
the structure of Horikawa varieties proved in this article and
{find different methods, as we do in Section 4}.

\medskip

{
Finally, in Section 5 we {show} that the main results of Sections 2 and 3
(precisely,
Theorems~\ref{regular}, \ref{theorem.p_g.bounded}, \ref{regularity} (1)
and \ref{topology}), when considered for strong Horikawa varieties,
fit in a broader setting, that is, they hold for
hyperelliptic subcanonical polarized  varieties. For instance,
if $(X,A)$ is  an $n$-dimensional hyperelliptic,
subcanonical polarized variety
 with
 canonical  singularities, then $X$ is simply connected and, if
 the image of the morphism induced by $|A|$ is singular, then
 $h^0(A) \leq n+4$.
 In particular, this holds if $(X,A)$ is a
    Calabi-Yau or Fano hyperelliptic polarized variety with
    canonical singularities.}

\section{Cohomological properties of Horikawa varieties}

We will use the following conventions and notations throughout the article.

\begin{convention}
 We will work over $\CC$. By a variety we will mean an irreducible variety.
\end{convention}

\begin{notation}\label{notation.section.regularity}
Let $X$ be an algebraic projective normal variety.
\begin{enumerate}
 \item  If
 $D$ is a Cartier divisor on $X$, we will use indistinctly
 the notations $H^0(\mathcal O_X(D))$ and $H^0(X,D)$
 (and  $h^0(\mathcal O_X(D))$ and $h^0(X,D)$) and $|D|$ will denote the
 complete linear series of $D$.

 \item We will denote by $K_X$  the
 canonical divisor of $X$.

 \item If $X$ is a Horikawa variety as in
 Definition~\ref{defi.Horikawa},
 then $\varphi$ will denote the canonical morphism of $X$,
 $Y$ will be the image of $\varphi$ (which is a variety of minimal
 degree) and
 $\overline X$ will denote
 the canonical model of $X$.
\end{enumerate}

\end{notation}

We now explore the topological and cohomological properties of Horikawa
varieties and its applications. It is quite interesting that just one numerical equality gives raise to so much geometry.
We start with a general study of the singularities that might occur in the context
of this article.

\begin{prop}\label{flat}
  If $X$ be a Horikawa variety,
then the Stein factorization of
the canonical morphism
\begin{equation*}
 \varphi: X \longrightarrow Y
\end{equation*}
of $X$  is
{
\begin{equation*}
 X \longrightarrow \overline{X} \overset{\phi}\longrightarrow Y,
\end{equation*}
where $\phi$ is the canonical map of $\overline{X}$, which is
therefore, a morphism.}
In particular, the Stein factorization of $X$
has at worst canonical singularities and $Y$
has at worst {log terminal singularities}.
\end{prop}

\begin{proof}
{Since $|K_{{X}}|$ is
base-point-free by \cite[Propositions 2.2, 2.5]{Kobayashi}},
it is known (see \cite[Theorem 2.1.15]{PAG}) that $|lK_{{X}}|$
 gives the Stein factorization $\widehat{X}$ of the canonical
 morphism of $X$,
 {for any
 sufficiently large integer
 $l$.  Then, for a suitable $l' \in \mathbb N$}
$$\mathrm{Proj}\Big(\bigoplus_{m \geq 0} H^0({X}, mK_{ X})\Big) =
{\mathrm{Proj}\Big(\bigoplus_{m \geq 0}
H^0({X}, l'mK_{ X})\Big)}
= \widehat{X},$$
{so the canonical model $\overline{X}$ of
${X}$ is nothing but $\widehat{X}$ and
$\phi$ is induced by $|K_{\overline X}|$.}
Therefore, $\widehat{X}$ has at worst canonical singularities.
It is a  well-known fact that Kawamata log terminal singularities
are preserved under finite maps (cf. \cite[Proposition 5.20]{KM}),
hence it follows that $Y$ has at worst {log terminal singularities}.
\end{proof}

\begin{notation}\label{notation.extra}
 {Given a Horikawa variety, we will
 keep the name $\phi$ used
 in Proposition~\ref{flat}
 for the canonical morphism of $\overline X$.}
\end{notation}

Recall (see \cite[Propositions 2.5]{Kobayashi})  that the canonical morphism
of a Horikawa variety is generically of degree $2$ onto its image. We then
say that a Horikawa variety is a \emph{strong Horikawa variety}
 if its canonical morphism is finite.
Proposition~\ref{flat} implies the following corollary for
strong Horikawa varieties:

\vskip -.5cm

{
\begin{corollary}\label{cor.strong.Horikawa}
A Horikawa variety $X$ is a strong Horikawa variety
if and only if its canonical model
is  $X$. In particular, the canonical model of a
Horikawa variety is a strong Horikawa variety.
\end{corollary}}

{The irregularity
is an important topological invariant of any variety, so
we study it in Theorem~\ref{regular}
for any Horikawa variety.}
{In our proof we make
 a reduction to the case of algebraic surfaces. However, the linear series involved is not canonical,
 unlike in the study of algebraic surfaces made by Horikawa.
Moreover, the singular case has to be handled carefully using a resolution of singularities.}
As a byproduct of our proof, we do obtain a more transparent proof of
Horikawa's results on surfaces as well.

\smallskip

{Although Theorem~\ref{regular} also follows
from Theorem~\ref{topology}, we give an independent proof below.
Among other things, the interest of this proof lies on the arguments employed.
These arguments are used in a crucial way in the proof of
Theorem~\ref{theorem.p_g.bounded}. For example, we build on them to show
that, for the intersection of $n-2$ hyperplane section of $Y$, which is
a Hirzebruch surface $\mathbb F_e$, $e$ is bounded. This is a key part in proving
the boundedness of $p_g$ for singular $X$.}

\begin{thm}\label{regular}
A  Horikawa variety  $X$ is regular.
\end{thm}

\begin{proof}
 If we show that $\overline{X}$ is regular, then so
is $X$.
From {Corollary~\ref{cor.strong.Horikawa} and}
\cite[Proposition 2.5]{Kobayashi},
we know that the base point free
complete canonical linear series
$|K_{\overline{X}}|$
induces a finite morphism
$\phi \colon \overline{X} \rightarrow
Y\subset \mathbb P^{N}$, of degree
$2$ onto a variety of minimal degree $Y$.
By taking $(n-2)$ general
hyperplane sections of $Y$ and then taking their
respective pullbacks to $\overline X$
we end up with a  {finite}
morphism $\phi_2: \overline X_2
\longrightarrow Y_2$, of degree $2$, where,
{by
\cite[Lemma 6.6]{CKM}},
$\overline X_2$ has canonical singularities
and $Y_2$ is a
surface of minimal degree.
We divide the situation into two cases:
when $Y_2$ is smooth and when $Y_2$ is singular.

\medskip

\noindent
{\bf Case 1.} $Y_2$ is a smooth surface of minimal degree.
\smallskip

Then $Y_2$ is either $\mathbb P^2$, the Veronese surface in $\mathbb P^5$ or a  Hirzebruch surface embedded as a scroll.
Since $\phi_{2}$ is flat,
${\phi_{2}}_{\ast}\mathcal{O}_{\overline{X}_2}=\mathcal{O}_{Y_{2}}\oplus
\mathcal{L}^{-1}$, where $\mathcal{L}^{\otimes2}=\mathcal{O}_{Y_{2}}(B)$,
with $B$ being the
branch divisor of the double cover.

\smallskip

First let $Y_2$ be a Hirzebruch surface embedded as scroll.
Then $Y_2$ is  embedded by a very ample linear series $|C_0+mf|$, where $C_0$ is a minimal section of $Y_2$. Then  $m>e$, where $C_0^2=-e$ and $e$ is the invariant of the Hirzebruch surface.
In this case note that, by the ramification formula for double covers, we have
\begin{equation*}
 \mathcal O_{\overline X_2}(K_{\overline{X}_2})=
\phi_{2}^{*}(\mathcal O_{Y_2}(K_{Y_2}) \otimes \mathcal L).
\end{equation*}
By adjunction we also have
\begin{equation*}
\mathcal O_{\overline X_2}(K_{\overline{X}_2})=
\phi_{2}^{*}(\mathcal O_{Y_2}(n-1)).
\end{equation*}
Since we are working on a rational surface,
{
\begin{equation}\label{ramification}
\mathcal O_{Y_2}(K_{Y_2})\otimes \mathcal L = \mathcal O_{Y_2}((n-1)C_{0}
+((n-1)m)f)
\end{equation}}
with $m>e$.
We
now show that $\overline{X}_{2}$ is regular in this case.
We have  $$H^1(\mathcal{O}_{\overline{X}_2}) = H^1(\mathcal{O}_{Y_2} )
\oplus H^1(\mathcal{L}^{-1}) =0,$$ because $Y_2$ is regular
and $H^1(Y_2, {\mathcal{O}_{Y_2}(K_{Y_2}) \otimes \mathcal{L}})$
vanishes (the latter follows from \eqref{ramification} and from  the
Leray spectral sequence applied to the projection from
the ruled surface $Y_2$ to $\mathbb{P}^1$).
Therefore $\overline
{X}_{2}$ is regular.

\smallskip

The case that $Y_2$ is $\mathbb{P}^2$ or the Veronese surface in $\mathbb{P}^5$ can be treated similarly.

\medskip
\noindent
{\bf Case 2.}  $Y_{2}$ is a singular surface of
minimal degree. \\
In this case $Y_{2}$ is a cone over a rational curve of degree $e \geq 2$. Let
$q \colon W_{2} \rightarrow Y_{2}$ be the minimal desingularization of $Y_2$. There
exists the following commutative diagram:

\begin{equation}\label{CD}
\begin{CD}
Z_{2} & @>{\overline{q}}>>  & \overline{X}_2\\
@Vp_2VV &  & @VV{\phi_2}V\\
W_{2} & @>q>> &  Y_{2},
\end{CD}
\end{equation}
where ${Z_{2}}$
is the normalization of the reduced part of the fiber
product $W_2 \times_{Y_{2}} \overline{X}_{2}$, which is irreducible.
The map $p_2 \colon Z_2 \to W_2$, which is a finite map of degree $2$,
and $\overline{q} \colon Z_2 \to \overline{X}_{2}$,
which is a birational map, are induced by the projections from the fiber product onto
each factor.
Since $W_2$ is smooth
and ${Z}_{2}$ is normal, $p_2$ is a flat morphism of degree 2,
and
$Z_{2}$ is Gorenstein.  Moreover,
${p_2}_{\ast}(\mathcal{O}_{{Z}_{2}})=\mathcal{O}_{W_{2}}\oplus
\mathcal{L}^{-1}$ where now $\mathcal{L}^{\otimes2}=\mathcal{O}_{W_{2}}(B)$ and
$B$ is the branch divisor of the double cover.
By adjunction
\begin{equation}\label{eq.pullback.adjunction}
 \phi_{2}^{\ast}(\mathcal{O}_{Y_{2}}(n-1))=
 \mathcal O_{\overline{X}_{2}}(K_{\overline{X}_{2}}).
\end{equation}
Let $y \in Y_{2}$
 be the
vertex of the cone. The surface $W_{2}$ is a Hirzebruch surface
with the minimal section $C_{0}^{2}=-e$, where $C_0=q^{-1}(y)$.
Let $F:=p_2^{-1}(C_0)$.
  There are two possible cases for $\phi_{2}^{-1}(y),$
either it consists of one point or it consists of two points.

\medskip

\noindent
{\bf Case 2.1.} $\phi_{2}^{-1}(y)$ is one point.

\smallskip
First we show $C_0$ is in the branch locus of $p_2$. Suppose the contrary.
If $C_{0}$ is not contained  in the branch locus of $p_2$, then $p_2^{\ast}C_{0}=F$ and $F^{2}=-2e$.
There exist
canonical
divisors $K_{Z_2}$ and $K_{\overline{X}_2}$
and a
nonnegative {integer} $a$
such that
$K_{Z_2}=\overline{q}^{\ast} K_{\overline{X}_2}+aF$.
Applying adjunction we have
\[
(K_{{Z_{2}}}+F)\cdot F=(
\overline{q}^{\ast}K_{{\overline X_2}}+(a+1)F)\cdot
F=-2e(a+1) \leq  -4,
\]
and this is impossible because $F$ is a reduced and
connected curve.

\medskip
Therefore  the minimal section  $C_0$ is in the branch locus of
$p_2$.
{
Since $C_0$ is in the branch locus of $p_2$, we have that
$F$ is isomorphic to $\mathbb P^1$}, that
$p_2^{\ast}(C_{0})=2F$ and that $F^{2}=-\frac{e}{2}$.
{We also have
$K_{{Z_{2}}}=\overline{q}^{\ast}(K_{\overline X_2})+aF$, with
$a$ nonnegative because $\overline X_2$ has canonical
singularities.}
{By \eqref{CD} and \eqref{eq.pullback.adjunction} we obtain,}
\begin{equation}\label{equation4}
K_{{Z_{2}}}=\overline{q}^{\ast}(K_{\overline X_2})+aF =
\overline{q}^{\ast} \phi_2^* \mathcal{O}_{Y_2}(n-1)+aF= p_2^*q^* \mathcal{O}_{Y_2}(n-1)+aF.
\end{equation}
Comparing \eqref{equation4}
with
\begin{equation*}
\mathcal O_{Z_2}(K_{Z_2})=p_2^*(\mathcal O_{W_2}(K_{W_2})
\otimes \mathcal{L}),
\end{equation*}
one sees that $\mathcal{O}_{Z_2}(aF) = p^*N$ for some line bundle $N$ on $W_2$.
Therefore, $p_2^*N^{\otimes 2} = \mathcal{O}_{Z_2}(2aF) =p_2^*\mathcal{O}_{W_2}(aC_0)$.
This implies in particular that $N^{\otimes 2}$ and $aC_0$ are numerically equivalent.
Since $W_{2}$ is a rational ruled surface, $N^{\otimes 2}$ and $aC_0$
are linearly equivalent. Therefore
$N \sim \mathcal{O}_{W_{2}}(a' C_{0})$ with
$2a'=a$, where $a'$ is an integer. Thus $a$ is a nonnegative even integer.
On the other hand, we have
\begin{equation}\label{ramification.formula}
K_{{Z_{2}}}= p_2^*K_{W_2}+R
\sim p_2^{\ast}(-2C_{0}-(e+2)f)+R,
\end{equation}
where $R$ is the ramification divisor of $p_2$.
From \eqref{equation4} and \eqref{ramification.formula}
we get
\begin{equation}\label{equation6}
 R\sim p_2^{\ast}((n+1)C_{0}
+(ne+2)f)+aF \sim p_2^{\ast}((ne+2)f)+(2n+2+a)F.
\end{equation}
 Since $Z_2$ is normal, $R$ can be written as
 $R=R_{1}+F$, where $R_{1}$ is a divisor  that does not
contains $F$ in its support.
 \begin{equation}\label{bound1}
 (a+1)e\leq 4.
 \end{equation}
 Since $a$ is even and $e\geq 2$, we have
 \begin{equation}\label{eq.for.boundedness1}
 a=0 \ \textrm{and} \  e=2, 3 \ \textrm{or} \ 4.
 \end{equation}
In particular, $\overline{q}$ is crepant.
{Then we have (see \eqref{equation6})}
\begin{equation*}
 R\sim p_2^{\ast}((n+1)C_{0}
+(ne+2)f).
\end{equation*}
Now, the ramification {formula} for $p_2$ reads
\begin{equation*}
 \mathcal O_{Z_2}(K_{Z_2})=p_2^*(\mathcal O_{W_2}(K_{W_2}) \otimes
\mathcal{L}),
\end{equation*}
 with
\begin{equation}\label{formula.branch.singular.case}
\mathcal{L}=\mathcal O_{W_2}((n+1)C_0+(ne+2)f).
\end{equation}
Since ${p_2}_{\ast
}(\mathcal{O}_{{Z_{2}}})=\mathcal{O}_{W_2}\oplus{\mathcal{L}^{-1}}$,
it follows from projection formula and duality that
$$h^1(\mathcal{O}_{Z_2})=
h^{1}({\mathcal{L}^{-1}}) =
h^{1}(\mathcal O_{W_2}((n-1)(C_{0}+ef)))=0.$$
Thus $Z_2$ is regular and hence so is $\overline{X}_{2}$
in this case.

\medskip

\noindent
{\bf Case 2.2.}
$\phi_{2}^{-1}(y)=\{x_{1},x_{2}\}$ are two distinct points.

\smallskip
 {{This implies that
$\phi_{2}$ is \'etale in the analytic neighborhood of $x_{1}$ and $x_{2}$.
Also, $p_2$ is \'etale on an analytic neighborhood of $C_{0}$.}}
 Let $E_{1}$ and
$E_{2}$ be exceptional divisors of $\overline{q}.$
We have $p_2^{\ast}
(C_{0})=E_{1}+E_{2},$ $E_{1}\cdot E_{2}=0$ and $E_{i}^{2}=-e=C_{0}^{2}$, with
$E_{i}\cong \mathbb{P}^{1}$.
Recall
$Z_2$ is Gorenstein as explained above.
We then have
\begin{equation*}
 K_{Z_2}=\overline{q}^{\ast}K_{\overline X_2}+a(E_{1}+E_{2}),
\end{equation*}
{with $a$ nonnegative.}
Applying
the adjunction formula we obtain:
\begin{equation}\label{bound2}
 -2=(K_{{Z_{2}}}+E_{i})\cdot E_{i}=(\overline{q}^{\ast}K_{X_2}
+a(E_{1}+E_{2})+E_{i})\cdot E_{i}=-e(a+1).
\end{equation}
Since $e\geq2$, then
\begin{equation}\label{eq.for.boundedness2}
 a=0 \ \textrm{and} \  e=2.
\end{equation}
This means that
$\overline{q}$ is crepant.
By a similar computation as in Case 2.1,
we have $H^{1}(\mathcal{O}_{\overline{X}_{2}})=0$.

\medskip

We thus conclude that $\overline{X}_{2}$
is regular in all cases.

\smallskip

{Since $\overline{X}_{2}$ is regular, so is $X_2$. Now we
prove  the Horikawa variety $X$ is regular.
Recall that ${X}_{2}$ is
obtained
as a complete intersection of members of the linear system
$|K_{{X}}|$.
Let $L_{i}=\mathcal O_{{X}_{i}}(K_{{X}}|_{{X}_{i}})$, where ${X}_{i}$ is the variety
obtained from  the  intersection of $n-i$
general members of $|K_{{X}}|$. Then
${X}_{i}$ has canonical singularities
by \cite[Lemma 6.6]{CKM}.
Consider the following short exact sequence:}
\begin{equation}\label{eq.restriction.sequence}
{0\longrightarrow L_{i}^{-1} \longrightarrow
\mathcal{O}_{{X}_{i}}
\rightarrow\mathcal{O}_{{X}_{i-1}}\longrightarrow 0.}
\end{equation}
{We note that $H^1(L_{i}^{-1})=0$
by Serre duality and
the  Kawamata-Viehweg vanishing theorem
so, if ${X}_{i-1}$ is a regular variety, then so is ${X}_{i}$.
We have shown that ${X}_2$ is regular,
so we have that ${X}_i$ is regular for all
$2 \le i \le n$ by induction.}
\end{proof}

{\begin{thm}\label{theorem.p_g.bounded}
 Let $X$ be a Horikawa variety of dimension $n$.
If  the image of $X$ by its canonical
 morphism is singular, {then  $p_g(X) \leq  n+4$}.
\end{thm}

\begin{proof}
 {
By \cite[Proposition 2.5]{Kobayashi}
we know that the base point free
complete canonical linear series $|K_{\overline{X}}|$
induces a finite morphism $\phi \colon \overline{X} \rightarrow
Y\subset \mathbb P^{N}$, of degree
$2$ onto $Y$, where $Y$ is a singular variety
of minimal
degree in $\mathbb P^{N}$.}
{Then $Y$ is a cone over a smooth variety of minimal degree.
Let $y$ be a point of the vertex of $Y$. If we choose  $n-2$
general hyperplanes through $y$, by considering their intersection and
its pullback by $\phi$, we get
a finite, degree $2$ morphism
$\phi_2: \overline X_2 \longrightarrow Y_2$, where
$Y_2$ is an irreducible, singular surface of minimal degree
(i.e., a cone over a
(smooth) rational normal curve) and $\overline X_2$ is
also irreducible}
{and
locally Gorenstein.}
{Since $Y_2$ has log terminal singularities, it follows
from \cite[Proposition 5.20]{KM}
that $\overline X_2$ also has log terminal singularities.
Then, since $\overline X_2$ is locally Gorenstein, $\overline X_2$ has canonical singularities.
Therefore $\overline X_2$ and $\phi_2$ are like $\overline X_2$ and $\phi_2$ in
the proof of Theorem~\ref{regular} and $Y_2$ is like $Y_2$ of Case 2 of
the proof of Theorem~\ref{regular}, so from \eqref{eq.for.boundedness1}
and \eqref{eq.for.boundedness2} if follows that
$Y_2$ is a nondegenerate surface in $\PPP^3$, $\PPP^4$ or $\PPP^5$.}
{Thus
$Y$ is a nondegenerate variety in $\PPP^{n+1}$,
$\PPP^{n+2}$ or $\PPP^{n+3}$, so
$p_g(X)=p_g(\overline X) \leq n+4$.}
\end{proof}

Now we present a very interesting result, Proposition~\ref{deformation},
on the deformations of canonical morphisms of Horikawa variety.
{This crucially depends on} {Theorem~\ref{regular}} and the base point freeness of $K_X$  .
{
First we make clear  what
we mean by a deformation of a morphism:}

{\begin{definition}\label{defi.deformation}
Let $X$ be  an algebraic projective normal variety  and let
\begin{equation*}
\psi: X \longrightarrow \PPP^N
\end{equation*}
be a morphism.
Let $T$ be
a smooth disc. A deformation of $\psi$ is
a $T$--morphism
\begin{equation*}
 \varPsi: \mathfrak X \longrightarrow \PPP^N_T
\end{equation*}
such that
\begin{enumerate}
 \item the variety $\mathfrak X$ is irreducible and reduced;
 \item the morphism $\mathfrak X \longrightarrow T$ is proper
and surjective;
\item $\mathfrak X_0=X$; and
\item  $\Psi_0=\psi$.
\end{enumerate}
\end{definition}}

\begin{prop}\label{deformation}
Let $X$ be a Horikawa variety of dimension $n$.
Let $\overline X$ be its canonical model. Then the general deformation
of the canonical morphism of $X$
is again a generically finite canonical morphism of degree $2$
onto a variety of minimal degree.
Also, the general deformation of $\overline X$
is again a canonical model of a
Horikawa variety.
\end{prop}

\begin{proof} We will use
\cite[Lemma 2.4]{Inventiones}. Note that, although
this statement
requires
smoothness, in fact,
it holds for
varieties with canonical singularities.
Thus, since by Theorem~\ref {regular}, {$X$ and
$\overline X$ are regular},
\cite[Lemma 2.4]{Inventiones} applies to both.
Therefore,
{for
any \emph{small}  deformation (by a small deformation we
mean that we shrink $T$ if needed)
\begin{equation*}
 \Phi: \mathfrak X \longrightarrow \PPP^N_T,
\end{equation*}
of the canonical morphism $\varphi$ of $X$
(respectively,
any  small deformation
\begin{equation*}
 \overline{\Phi}: \overline{\mathfrak X} \longrightarrow \PPP^N_T,
\end{equation*}
of the canonical morphism $\phi$ of $\overline X$),
$\Phi_t$ (respectively, $\overline{\Phi}_t$) is  a canonical morphism, for
all $t \in T$.}
{In addition,
the image of $\mathfrak X_t$ (respectively,
of $\overline{\mathfrak X}_t$)
under its canonical morphism $\Phi_t$
(respectively, its canonical morphism $\overline{\Phi}_t$)
is of
maximum dimension $n$.}
By the main theorem of \cite{Kawamata}, the
deformation of canonical singularities is again canonical
and, by \cite[Theorem 6]{Kawamata}, $K_{\mathfrak X_t}^{n}$
and $p_g(\mathfrak X_t)$
(respectively, $K_{\overline{\mathfrak X}_t}^{n}$ and
$p_g(\overline{\mathfrak X}_t)$)
are invariant under deformations.
{Thus, $\mathfrak X_t$
(respectively, $K_{\overline{\mathfrak X}_t}^{n}$) is a
Horikawa variety and, by \cite[Proposition 2.5]{Kobayashi},
$\Phi_t$ is a generically finite morphism of degree $2$
(respectively, $\overline{\Phi}_t$ is a finite morphism of degree $2$)
onto
a variety of minimal degree.
In particular $\overline{\mathfrak X}_t$
is a strong Horikawa variety so,
by Corollary~\ref{cor.strong.Horikawa}, is its own canonical model.}
\end{proof}

{Proposition~\ref{deformation} has
this obvious implication for the components
of the moduli space of varieties of general type:}

\begin{corollary}\label{cor.moduli}
 {The general points of the components of the moduli of varieties of general type that contain a canonical model of a Horikawa variety are canonical models of Horikawa varieties.}
\end{corollary}

We now {recall} the generalized notion of regularity for an algebraic variety,
for varieties of arbitrary dimension.

\begin{defn}
A variety $X$ of dimension $n$ is said to be
{\it pluriregular} if $H^{1}(\mathcal{O}_{X})=\cdots=H^{n-1}(\mathcal{O}_{X})=0.$
\end{defn}

\begin{thm}\label{regularity}
 Let ${X}$ be a Horikawa variety of dimension $n$.
 \begin{enumerate}
  \item If the image $Y$ of its
canonical morphism $\varphi$ is smooth,
then ${X}$ is pluriregular.
\item If, in addition, $Y$ is a
rational normal scroll
then the general fiber  of ${X}$
{over $\mathbb P^1$ is a Horikawa variety}
with geometric genus
$n$.
 \end{enumerate}
\end{thm}

\begin{proof}
{For the proof of (1) we may assume $n >2$,
for the result for $n=2$ has been proved in Theorem~\ref{regular}.}
In view of Proposition~\ref{flat},
we will work with the canonical model
$\overline X$ of ${X}$, {since
${X}$ is pluriregular if and only if $\overline X$
is pluriregular.}
{The canonical morphism $\phi$ of
$\overline X$ is finite of degree $2$.
Since $\overline X$ is locally Cohen-Macaulay}
and $Y$ is smooth, the morphism
$\phi$  is flat.
Then
{\begin{equation*}
 \phi_*\mathcal{O}_{\overline X} = \mathcal{O}_Y \oplus \mathcal{L}^{-1},
\end{equation*}
where $\mathcal{L}$ is line bundle.}
Let $R$ be the ramification divisor and let $B$
be the branch divisor of {$\phi$}.
Then  $\mathcal{O}_Y(B)= \mathcal{L}^{\otimes 2}$ and
\begin{equation}\label{equation.regularity}
 \mathcal O_{\overline X}(K_{\overline X})=
 \phi^*(\mathcal{O}_Y(K_Y) \otimes \mathcal{L}) =
 \phi^*\mathcal{O}_Y(1).
\end{equation}
{Showing}  that $\overline X$ is
pluriregular is equivalent to showing the vanishing of
the intermediate cohomology of $\mathcal{O}_Y$ and $\mathcal{L}^{-1}$,
because
$R^i{\phi}_* \mathcal{O}_{\overline X}=0$
for all $i >0$ {($\phi$ is finite)}.
By the classification of varieties of minimal degree, we have
three possible cases for $Y$:

\begin{enumerate}
\item  $Y$ is $\mathbb{P}^{n}$, $n\geq 3$.

\item $Y$ is a smooth quadric hypersurface in $\mathbb{P}^{n+1}$, $n\geq 3$.

\item  $Y$ is a smooth rational normal scroll of dimension $n\geq 3$.
In this case
$Y$ is
fibered over $\mathbb{P}^{1}$, hence $\overline X$ is also fibered over
$\mathbb{P}^{1}$
\end{enumerate}

{If $Y$} is $\mathbb{P}^n$ or a smooth quadric in $\mathbb{P}^{n+1}$,
then  the vanishing of the intermediate cohomology of $\mathcal{O}_Y$ and $\mathcal{L}^{-1}$
follows from the vanishing of the intermediate cohomology of line bundles in
$\mathbb{P}^n$  and $\mathbb{P}^{n+1}$ respectively.

\smallskip

It remains to consider the case that $Y$ is a smooth rational normal scroll,
which is a $\mathbb{P}^{n-1}$-bundle over $\mathbb{P}^1$.
 We may write
 \begin{equation*}
 p \colon Y=\mathbb{P}(\mathcal{E}) \longrightarrow \mathbb{P}^1,
 \end{equation*}
 where
 \begin{equation*}
 \mathcal{E}=\mathcal{O}_{\mathbb{P}^1}(a_1) \oplus
 \cdots \oplus \mathcal{O}_{\mathbb{P}^1}(a_n)
 \end{equation*}
 with $0 \, {<} \, a_1 \le \ldots \le a_n$.
{We have}  $$h^i(\mathcal O_{\overline X})=
h^{i}({\phi}_{\ast}\mathcal{O}_{\overline X})=
h^{i}(\mathcal{O}_Y) + h^{i}(\mathcal{L}^{-1})=h^{i}
(\mathcal{O}_Y)+h^{n-i}(\mathcal{O}_Y(1)),$$
because
$\mathcal{O}_Y(K_{Y}) \otimes \mathcal{L}=\mathcal{O}_{Y}(1)$
{by
\eqref{equation.regularity}}.
For any $0 < i < n$ $$h^{i}
(\mathcal{O}_Y)=h^i(p_*\mathcal{O}_Y) { =
h^i(\mathcal{O}_{\mathbb{P}^1})}=0 $$ and
$$h^{i}(\mathcal{O}_Y(1))=h^i(p_*\mathcal{O}_Y(1))
=h^i(\mathcal{E})=0.$$ The pluriregularity of
$\overline X$ now follows.

\smallskip
Let us now prove (2).
Let $F$ be a general fiber of $p$ and let
$G'=\phi^{-1}(F)$.
 We have the
following long exact sequence of cohomology
\begin{multline*}
0\rightarrow H^{0}(\mathcal O_{\overline X}(K_{\overline X}-G'))
\longrightarrow
H^{0}(\mathcal O_{\overline X}(K_{\overline X}))\longrightarrow
\\
H^{0}%
(\mathcal O_{G'}(K_{\overline X}|_{G'}))\longrightarrow
H^{1}(\mathcal O_{\overline X}(K_{\overline X}-G')).
\end{multline*}
Because of  {$R^i{\phi}_*\mathcal{O}_{\overline X}
=0$},
$\mathcal O_{\overline X}(K_{\overline X}-G')=
\phi^{\ast}(\mathcal{O}_{Y}(1)
\otimes \mathcal{O}_{Y}(-F))$ and
{\eqref{equation.regularity}},
we have
\begin{multline*}
 h^{1}(\mathcal O_{\overline X}(K_{\overline X}-G'))
 =h^{1}(\mathcal{O}_{Y}(1)
 \otimes \mathcal{O}_{Y}(-F))+
h^1(\mathcal{L}^{-1} \otimes \mathcal{O}_{Y}(1)
\otimes \mathcal{O}_{Y}(-F))= \\
h^1(\mathcal{O}_{Y}(1) \otimes \mathcal{O}_{Y}(-F))+
h^{n-1}(\mathcal{O}_{Y}(2) \otimes \mathcal{O}_{Y}(-F)).
\end{multline*}
Pushing forward to $\mathbb{P}^1$ via $p$, it is easy to see that
{both
$h^1(\mathcal{O}_{Y}(1) \otimes \mathcal{O}_{Y}(-F))$ and
$h^{n-1}(\mathcal{O}_{Y}(2) \otimes \mathcal{O}_{Y}(-F))$ vanish}.
This implies that the complete linear
series $|K_{\overline X}|$ restricts to the complete linear series
$|K_{\overline X}|_{_{{G'}}}|$, which is
$|K_{G'}|$. Hence
$|K_{G'}|$ maps $G'$ onto $G=\mathbb{P}^{n-1}$
{as a finite double cover.
By \cite[Theorem 5.20]{KM} and since $K_{G'}$ is Cartier, we have that
$G'$ has canonical singularities.
Therefore $G'$ is a strong Horikawa variety} and $p_{g}(G')=n$.
{Let $G$ be the fiber of $X$ that corresponds to $G'$.
Then $p_g(G)=p_g(G')$.}
{Then the restriction of $\varphi$ to $G$ is the canonical
morphism of $G$,}
{which is generically of degree $2$ onto
$\mathbb{P}^{n-1}$.
Since $G'$ has canonical singularities, so does $G$,}
{therefore
$G$ is a Horikawa variety.}
\end{proof}

\begin{rem}
It might be illuminating to observe how beautifully the situation fits with
the classical case of surfaces. In Horikawa's work in \cite{Horikawa},
it turns out that
smooth Horikawa surfaces are indeed regular genus two fibrations over
$\mathbb{P}^{1}$ (in this case, $G'$ is a curve so $p_{g}(G')$
is the genus of $G'$). In the higher
dimensional case Theorem~\ref{regularity} shows the exact analogy.
That the numerical
situation dictates this in all dimensions is indeed compelling.
\end{rem}

\section{Simple connectedness of Horikawa varieties}

We devote this section to proving the
simple connectedness of Horikawa varieties. For this, first
we need to
state two results.
The first one is follows from a well-known result
on the fundamental group of the complement of smooth
submanifold of (real) codimension more than $1$.
The second one  is
based in the ideas and results
of  M. Nori
in \cite{Nori}.

\begin{lem}\label{cor.Godbillon}
 {Let $V$ be a smooth complex variety and let
 $W$ a complex, reduced
 subvariety of $V$, not
 necessarily smooth, and let
  \begin{equation*}
   \pi_1(V \smallsetminus W) \longrightarrow \pi_1(V)
  \end{equation*}
  be the homomorphism
  induced by inclusion.
  This homomorphism is surjective if
  $W$ has codimension $1$ in $V$ and an isomorphism if
  $W$ has codimension more than $1$ in $V$.}
 \end{lem}

\begin{lem}\label{Nori2}
{
Let $p:V' \longrightarrow V$ be a finite morphism
of degree $2$
 among smooth, (quasiprojective) complex varieties $V'$ and $V$
 of dimension
 $n \geq 2$,
branched along a smooth, irreducible divisor $B$.
 If $V$ is
 simply connected, then so is $V'$.}
\end{lem}

\begin{proof}
{Consider the commutative diagram
 \begin{equation}\label{diagram.B.square}
  \xymatrix@R-10pt{
   1 \ar[r] & K' \ar[d]^\iota \ar[r]  &
   \pi_1(V' \smallsetminus p^{-1}(B))
   \ar[d]^{\iota'} \ar[r] & \pi_1(V')
  \ar[d]_j \ar[r] & 1 \\
   1 \ar[r] & K \ar[r] & \pi_1(V \smallsetminus B)
   \ar[r] & \pi_1(V) \ar[r] & 1,}
 \end{equation}
 where
  the horizontal exact sequences are exact
 at the right hand side by \linebreak
 Lemma~\ref{cor.Godbillon}.
 Since $p$ restricted to $V' \smallsetminus p^{-1}(B)$
 is a $2$-sheet unramified cover,
 $\iota'$ is injective
 (and so is $\iota$) and
 its cokernel  is isomorphic to $\mathbb Z_2$.
 Since $V$ is simply connected, chasing the diagram
 we get the following short exact sequence
 \begin{equation*}
  0 \longrightarrow L \longrightarrow K/K'
  \longrightarrow \mathbb Z_2 \longrightarrow 0,
 \end{equation*}
 where $L$ is the kernel of $j$.}

{By \cite[1.2, 1.4 B]{Nori}, $K$ is a cyclic group,
 so is $K'$ and the generator of $K'$ maps
 to the square of the generator of $K$. Therefore
 $K/K'$ is also isomorphic to $\mathbb Z_2$, so
 $L=0$.}
\end{proof}

 \smallskip

{We further explore the topological properties of
the Horikawa varieties. We prove that any Horikawa variety is
simply connected.}
The similarities with the surface case is striking.
The fact that just two invariants $p_g(X)$ and $K_{X}^n$ determine
the topology entirely is rather remarkable.
 {Note that
 a strong Horikawa variety together with its
canonical divisor is a hyperelliptic polarized variety as defined
by Fujita in \cite[Definition 1.1]{Fujita}
(see also \cite[Definition 0.1]{BGG}). Then
 Theorem~\ref{topology}
extends what
\cite[Corollary 5.17]{Fujita} says with respect to the
simple connectedness of strong Horikawa varieties. Indeed,
\cite[Corollary 5.17]{Fujita} states the simply connectedness
only of
strong Horikawa varieties whose canonical
morphism is a double cover of a (smooth) rational normal scroll
with connected branch divisor,
while  Theorem~\ref{topology}
states that all Horikawa varieties are simply connected.}

\begin{thm}\label{topology}
Any Horikawa variety $X$
is simply connected.
\end{thm}

\begin{proof}
We will work with the canonical model $\overline X$
of $X$ and,
by taking the intersections of $n-2$ general
hyperplane sections, from the canonical morphism $\phi$
of
$\overline X$ we obtain a finite morphism
\begin{equation*}
 \phi_2: \overline X_2 \longrightarrow Y_2,
\end{equation*}
 of degree $2$
onto a surface of minimal degree $Y_2$, as we did in
the proof of Theorem~\ref{regular}.
{Our goal now is to prove
that $\overline X_2$ is simply connected.} We split the argument in
several cases.

\smallskip
\noindent {\bf Case 1:} $\overline X_2$ and $Y_2$ are smooth.
Therefore $Y_2$ is either
$\mathbb {P}^{2}$,
the Veronese surface in $\mathbb P^5$ or a
Hirzebruch surface embedded as a rational normal scroll. Let $B$ be the
branch divisor of $p$. Since $\overline X_2$ is smooth, so is $B$.
The case in which $B$ is ample, that includes
the cases of $Y_2$ being $\mathbb {P}^{2}$ or
the Veronese surface in $\mathbb P^5$, is straightforward.
{Indeed, by \cite[Corollary 2.7]{Nori},
$\overline X_{2}$ is simply connected.}

\smallskip
Now we deal with the case of $Y_2$ being a rational normal
scroll and $B$ not ample.
Then $Y_2$ is a Hirzebruch surface embedded by
$|C_0+mf|$, where $C_0$ is its minimal section,
 $f$ is a fiber, $C_0^2=-e$ and $m \geq e+1$.
is a Hirzebruch surface.
By \eqref{ramification},
$$B \sim 2(n+1)C_{0}+2((n-1)m+e+2)f).$$
{Let
$B \sim \alpha C_0 + \beta f$.}
{Since $B$  is smooth,
then either $\beta=\alpha e$ or $\beta=(\alpha-1)e$.}
{If $\beta=\alpha e$, then $B$
is big and base point free,}
{so $B$ is irreducible.
Since $Y_2$ is simply connected, by Lemma~\ref{Nori2}, so is $\overline X_2$.}

\smallskip
{If $\beta=(\alpha-1)e$,
then $B=C_0+B_1$, with $B_1$ big and base point free
(thus $B_1$ is smooth and irreducible)}
{and
$C_0$ and $B_1$ disjoint. Since $Y_2 \smallsetminus C_0$
is an $\mathbb A^1$-fibration over $\mathbb P^1$,
$Y_2 \smallsetminus C_0$
is simply connected.}
{
By Lemma~\ref{Nori2}, $\overline X_2
   \smallsetminus p^{-1}(C_0)$ is also simply connected
 and,
 by Lemma~\ref{cor.Godbillon},
so is $\overline X_2$.
Thus we have showed that, if $Y_2$ is a smooth rational normal
scroll and $B$ is not ample, then $\overline X_2$ is also
simply connected.}

 \smallskip

\noindent
{\bf Case 2:} $\overline X_2$ is singular
and $Y_2$ is smooth.\\
{Then, as in Case 1
of the proof of Theorem~\ref{regular},
$\phi_2$ is a flat
double cover determined
by a branch divisor $B$ (which is necessarily singular).}

\smallskip
\noindent {
{\bf Case 2.1.} $Y_2$ is  $\mathbb P^2$
or a Veronese surface in $\mathbb P^5$.}

\noindent
{Although
 $B$  is  singular, since $B$ is base--point--free,
a general member of $|B|$ is smooth and, in fact,
we can consider a
deformation $\mathcal B$ of $B$ over a disc $T$, such that
$\mathcal B_0=B$, $\mathcal B$ is smooth
and $\mathcal B_t$ is smooth for all $t \neq 0$.
Using $\mathcal B$ as a relative branch divisor, we can
construct a deformation
\begin{equation*}
\mathcal X \longrightarrow Y_2 \times T,
\end{equation*}
 flat over $T$  of
$\phi_2$,
 such that $\mathcal X_0=\overline X_2$,
 $\mathcal X$ is smooth
and $\mathcal X_t$ is smooth for all $t \neq 0$.}
{Shrinking $T$ if
necessary,
by \cite[Theorem I.8.8]{BHPV} and \cite[Lemma 1.5.C]{Nori},}
{$\pi_1(\mathcal X_t)$ surjects onto
$\pi_1(\mathcal X_0)$ for any general $t$ in $T$.
Since $\mathcal X_t$ is a finite double
cover of $Y_2$ and $\mathcal X_t$ is smooth, by the
arguments used in Case 1, $\mathcal X_t$
is simply connected, and so is $\overline X_2$.}

\smallskip

\noindent
{{\bf Case 2.2.} $Y_2$ is
a rational normal scroll.}

\noindent {Then $Y_2$
is a Hirzebruch surface.}
{If
$B$ is big and base--point--free, then the general member of
$|B|$ is smooth. Therefore, arguing as in Case 2.1, we can deform
$\overline X_2$ to a smooth, simply connected surface
$\mathcal X_t$, so $\overline X_2$ is also simply
connected.}
{Now assume $B$ is not big and base--point--free and,
 with the same notation of Case 1,
 let $B \sim \alpha C_0 + \beta f$.
 Then, by \eqref{ramification}, $\beta < \alpha e$.}
Since $X_2$ is normal,  $C_0$ cannot occur in the branch locus
with multiplicity more than $1$, so $B\thicksim B' +C_0$,
{$\beta \geq (\alpha-1)e$,}
and $B' \cdot C_0 \geq 0$.
Then, {by \eqref{ramification},}
$B'$ {is big and base--point--free and}
any general member
in $|B'|$ is irreducible
and nonsingular.
In addition, since $H^1(\mathcal O_Y(B'-C_0))=0$,
{we may assume that any general member
of $|B'|$
 intersects $C_0$
transversally.}
{Thus there is a family of divisors $\mathcal B_t$ and $\mathcal B'_t$
of $Y_2$ over a disc $T$ such that $\mathcal B_0=B$,
$\mathcal B_0'=B'$ and, for all $t \in T$, $t \neq 0$,
the divisor $\mathcal B'_t$ is smooth and meets $C_0$ transversally.
Let
\begin{equation*}
\Phi: \mathcal X \longrightarrow Y_2 \times T,
\end{equation*}
be the double cover of $Y_2 \times T$, branched along the total space
of the family formed by the divisors $\mathcal B_t$.
As in Case 2.1, $\Phi$ is  a deformation,
flat over $T$, of
$\phi_2$.
Let
\begin{equation*}
 \Psi: \widetilde{\mathcal X} \longrightarrow \mathcal X
\end{equation*}
be the minimal desingularization of $\mathcal X$.}
Since $C_0$ is the section
of the rational ruled surface $Y_2$,
$Y_{2}\smallsetminus C_0$ is an
${\mathbb A^1}$-fibration
over  ${\mathbb P^1}$,
hence  $Y_{2}\smallsetminus C_0$ is simply connected.
{Then, by Lemma~\ref{Nori2},
$\mathcal X_t \smallsetminus\Phi_t^{-1}(C_0)$ is also
simply connected.
Since $\widetilde{\mathcal X}_t \smallsetminus
\Psi^{-1}(\Phi_t^{-1}(C_0))$
and $\mathcal X_t \smallsetminus\Phi_t^{-1}(C_0)$ are isomorphic,
we have $\widetilde{\mathcal X}_t \smallsetminus
\psi^{-1}(\Phi_t^{-1}(C_0))$
is simply connected and, by Lemma~\ref{cor.Godbillon},
so is $\widetilde{\mathcal X}_t$.
Then, by \cite[Lemma 1.5.C]{Nori},
$\pi_1(\widetilde{\mathcal X}_t)$ surjects onto
$\pi_1(\widetilde{\mathcal X})$, so $\widetilde{\mathcal X}$ is also
simply connected. Since
$\mathcal X$ has canonical singularities,}
{by \cite[Theorem 1.1]{Takayama},
$\mathcal X$ is simply connected.
Shrinking $T$ if
necessary,
by \cite[Theorem I.8.8]{BHPV},
$\pi_1(\mathcal X)$ and
$\pi_1(\mathcal X_0)$ are isomorphic, so
${\mathcal X}_0=\overline X_2$
is simply connected.}

\smallskip
\noindent
{\bf Case 3.} $Y_2$ is singular.
In this case we consider the desingularization diagram~\eqref{CD}.
{Since the singularities
of $\overline X_2$ and $Z_2$ are canonical,}
{it follows from   \cite[Theorem 7.8]{Kollar}
or \cite[Theorem 1.1]{Takayama}
that
$\pi_1(\overline X_2)$ and $\pi_1(Z_2)$ are isomorphic.}
We consider now the double
cover
\begin{equation*}
  p:Z_{2}\longrightarrow W_{2}
\end{equation*}
in the diagram~\eqref{CD}.
As calculated in Theorem~\ref{regular},
the branch divisor $B$ of $p$
satisfies
\begin{equation*}
B \sim 2(n+1)C_0+2(ne+2)f.
\end{equation*}
 {If $e=2$, then
$B$ is big and base--point--free, so the general member of $|B|$ is smooth.
Arguing like in Case 2.1 we conclude that $Z_2$ is simply
connected and,
as observed above, so is $\overline X_2$.
If $e > 2$,} since $Z_2$ is normal,
$C_0$ occurs with multiplicity $1$
in $B$. This shows that $B=B'+C_0$,
where $B'$ is linearly equivalent to $(2n+1)C_0+2(ne+2)f$.
As we saw in \eqref{eq.for.boundedness1} and
\eqref{eq.for.boundedness2}, we have
$e \leq 4$. Then $B'\cdot C_0 \geq 0$,
$C_0$ is the fixed part of $|B|$ {and
$B'$ is big and base--point--free, so the general member
of $|B'|$ is a smooth irreducible curve.
In addition, $H^1(B' -C_0)=0$, so the general member
of $|B'|$ meets $C_0$ transversally.
Then arguing as in Case 2.2
we conclude
that $Z_2$ is simply connected and so is $\overline X_2$.}

\smallskip
{Thus we have seen that, in all three cases,
$\overline X_2$ is
simply connected.
Then, by Lefschetz hyperplane section theorem, $\overline X$ is
also simply connected and, by \cite[Theorem 1.1]{Takayama},
so is $X$.}
\end{proof}

\section{Birationality and projective normality}

Recall that,
if $X$ is a Horikawa variety, then
$K_{X}$ is base point free and the canonical morphism of $X$ maps $X$
onto a variety of minimal degree.
We now study the pluricanonical morphisms of Horikawa
varieties. We want to know when a pluricanonical morphism of $X$ is birational and maps
$X$ to a projectively normal variety. More precisely, we want to know when a pluricanonical morphism
of the canonical model of $X$ is an embedding and its image is projectively normal.

\smallskip
We will use the following notation throughout this section:

\begin{notation}\label{notation.section.proj.normality}
Let $X$ be a
Horikawa variety of dimension $n$.

{According to the context, $X_n$ will be $X$ or $\overline X$}.
Let $X_{1}\subset\cdots\subset X_{n^{\prime}}\subset\cdots\subset
X_{n}$ be
irreducible subsequent $n^{\prime}$-dimensional, complete
intersections of {general} members of
$|K_{X}|$ {or
of $|K_{\overline X}|$}.
Note that $X_{1}$ is a {smooth and irreducible} curve,
let us call it $C$.
{According to the context,}
we denote the line bundle, on $X_{n'}$ or on $\overline X_{n^{\prime}}$,
associated to
$K_{X}|_{X_{n^{\prime}}}$ or
$K_{\overline X}|_{\overline X_{n^{\prime}}}$,
by $L_{n^{\prime}}$. In
this notation, $L_{n}$ is $\mathcal O_X(K_{X})$ or
$\mathcal O_{\overline X}(K_{\overline X})$ accordingly.
Finally, let
\[
H^{0}(L_{n^{\prime}}^{\otimes s})\otimes
H^{0}(L_{n^{\prime}}^{\otimes t})\xrightarrow{\ \alpha(s,t;n^{\prime})\ }
H^{0}(L_{n^{\prime}}^{\otimes s+t})
\]
be the usual multiplication map of global sections.
\end{notation}

\bigskip

\begin{prop}\label{neverbirational}
Let $X$ be a Horikawa variety of dimension $n$.
 {If $1 \leq s \leq n$ or if $s=n+1$ and  $p_{g}(X)=n+1$, then $|sK_{X}|$
 does not induce
 a
 birational morphism.
}
\end{prop}

\begin{proof}
{From \eqref{eq.restriction.sequence}
we obtain, for all $2 \leq n' \leq n$,}
\begin{equation}\label{eq.restriction.sequence2}
 {0 \longrightarrow \mathcal O_{X_{n'}} \longrightarrow
 L_{n'} \longrightarrow L_{n'-1} \longrightarrow 0.}
\end{equation}
{From \eqref{eq.restriction.sequence}, in the proof of
Theorem~\ref{regular} we also saw that,
for all $2 \leq n' \leq n$, $X_{n'}$ is regular. This together with
\eqref{eq.restriction.sequence2} implies
\begin{equation*}
h^0(L_1)=h^0(\mathcal O_X(K_X))-n+1.
\end{equation*}
Since $X$ is a Horikawa variety,
\begin{equation*}
\mathrm{deg}L_1=2h^0(L_1)-2,
\end{equation*}
so, by Clifford's theorem,
$C$ is hyperelliptic and $L_1$ is a multiple of the $g^1_2$ of $C$.
By adjunction, $L_1^{\otimes n}=\mathcal O_C(nK_{X}|_{C})=\mathcal O_C(K_{C})$ so,
if $1 \leq s \leq n$, then $L_1^{\otimes s}$ is in the
special range of $C$ and
$|L_1^{\otimes s}|$ induces a morphism of degree $2$.
If $p_g(X)=n+1$,
then $\deg L_1=K_X^n=2$, so $L_1$ is the $g^1_2$ of $C$.
By adjunction we have $L_1^{\otimes n+1}=\mathcal O_C((n+1)K_{X}|_{C})=
\mathcal O_C(K_{C})\otimes L_1$.
Since $C$ is hyperelliptic, $|L_1^{\otimes n+1}|$ also induces
a morphism of degree $2$ in this case.
Therefore, under the hypothesis of the statement,
the restriction of the $s$-canonical morphism of $X$ to $C$ has
degree $2$. Since
$C$ is the intersection of $n-1$ general divisors of $|K_X|$, there is an open set of $X$ where
the $s$-canonical morphism of $X$ has degree $2$, hence it is not birational.}
\end{proof}

\medskip Before we deal with the birationality and projective normality of
pluricanonical systems of a Horikawa variety, we will state a particular case of
a very useful more general result (see \cite[Observation 1.2]{Calabi-Yau}).

\begin{lem}\label{divideandrule}
Let $X$ be a projective variety, $E$ a coherent sheaf, $\mathcal{L}_1$, $\mathcal{L}_2$,..., $\mathcal{L}_r$ be line bundles on $X$.
Let $$\gamma \colon H^0(E)\otimes H^0(\mathcal{L}_{1}\otimes \mathcal{L}_{2}\otimes\cdots\otimes \mathcal{L}_r) \longrightarrow H^0(E\otimes \mathcal{L}_{1}\otimes \mathcal{L}_{2}\otimes\cdots\otimes \mathcal{L}_r)$$ be the multiplication map of global sections.
If the multiplication maps
$$\alpha_j \colon H^0(E\otimes  \mathcal{L}_1\otimes \cdots \otimes \mathcal{L}_{j-1})\otimes H^0(\mathcal{L}_{j}) \longrightarrow H^0(E\otimes \mathcal{L}_1\otimes \mathcal{L}_2\otimes \cdots \otimes \mathcal{L}_{j})$$ are surjective for all $1 \le j \le r$,
then $\gamma$ is surjective.
\end{lem}

Lemma~\ref{divideandrule} will be crucial in
proving Proposition~\ref{easyprojnormal}, which is
the following general result on projective normality
{of pluricanonical images of varieties of general type.
Proposition~\ref{easyprojnormal} shows that
our main result} on projective normality {of pluricanonical images of Horikawa varieties},
Theorem~\ref{projnormal}, has the correct bounds as far as geometry of Horikawa varieties
is concerned. {Proposition~\ref{easyprojnormal} also shows}
that general standard methods {using}
Castelnuovo-Mumford regularity
{(see Remark~\ref{easyremark}) yield limited results:}
\medskip

\begin{prop}\label{easyprojnormal}
Let $X$ be a variety of general type of dimension $n$
{with  Gorenstein}
{canonical singularities and}
base-point-free canonical bundle.
Then
the image of $X$ by the morphism induced
by $|sK_{X}|$ is a projectively normal variety for all $s\geq n+2$.
\end{prop}

\begin{proof}
{Let $L \colon =\mathcal O_X(sK_{X})$.}
It is enough to show that
\begin{equation*}
 {H^0({L})^{\otimes r+1}}
 \overset{\gamma_r}\longrightarrow
H^0({L^{\otimes r+1}})
\end{equation*}
is surjective {for all $r\geq 1$.
By} Lemma~\ref{divideandrule}, it is enough to show
that the maps
$${\alpha_j: H^0(X, jK_X )\otimes H^0(X, K_X) \longrightarrow H^0(X, (j+1)K_X)}$$
are surjective for all {$j \geq s$}.
Note that since $s \ge n+2$ then {$s-i \ge 2$ for all $1 \le i \le n$}.
It follows from  the
Kawamata-Viehweg vanishing
{that
$H^i(X, (s-i)K_X)=0$ for all $1 \le i \le n$, so $\mathcal{O}_X(K_X)$
is $s$-regular}.
By  Castelnuovo-Mumford Lemma (see \cite[p. 41, Theorem 2]{Mumford}),
the surjection of {$\alpha_j$ for all $j \ge s$} now follows.
Hence by Lemma~\ref{divideandrule}, the multiplication map
{$\gamma_r$ surjects for all $r\geq 1$.}
\end{proof}

\begin{rem}\label{easyremark} What is at the heart of the proof of the above
Proposition~\ref{easyprojnormal} is the fact that the multiplication map
\begin{equation*}
 \alpha_s: H^0(X,sK_{X}) \otimes H^0(X,K_{X}) \longrightarrow H^0(X,(s+1)K_{X})
\end{equation*}
 is surjective,
for all $s\geq n+2$, for a variety of general type of dimension $n$. This follows by  the
Castelnuovo-Mumford Lemma and the Kawamata-Viehweg vanishing as noted in the proof of the
Proposition~\ref{easyprojnormal}.
{Note that, if $\alpha_{n+1}$ were surjective,
Castelnuovo-Mumford  Lemma would be of no help whatsoever to prove it.}
\end{rem}

In view of Remark~\ref{easyremark} the following result on projective normality
needs non standard methods of proof.
{The proof of Theorem~\ref{projnormal}
crucially
uses the  structure of Horikawa varieties, namely,
that its canonical map is a morphism which is a
generically double cover a variety of minimal degree and
Theorem~\ref{regular}.}

\begin{thm}\label{projnormal}
Let ${X}$ be a Horikawa variety of dimension $n$.
Then $K_{\overline X}^{\otimes s}$ embeds the
canonical model $\overline X$ of ${X}$
as a projectively normal variety
if and only if {$s\geq
n+2$ or $s=n+1$ and $p_g({X}) \neq n+1$}.
\end{thm}

\begin{proof}
{The ``only if'' part of the statement follows
from Proposition~\ref{neverbirational}.}

\smallskip
We will prove the much more subtle {``if'' part
of the statement}
in several steps. We define $X_{n'}$ and $L_{n'}$ starting
from $\overline X$, as explained in
Notation~\ref{notation.section.proj.normality}.
We set $L=K_{\overline X}^{\otimes s}$.

\smallskip
\textit{Step 1:}
Since $\overline X$ is regular by Theorem~\ref{regular},
the varieties $X_{n^{\prime}}$
  are all regular
  {(in fact, this was proved at the
  end of the proof of Theorem~\ref{regular})}.
{It follows by using adjunction inductively that, in addition,
the varieties $X_i$ are of general type}.

\bigskip

\textit{Step 2:} To prove the theorem, it would be enough to show that
the multiplication map{s}
\[
H^{0}(L^{\otimes t})\otimes H^{0}({L})
\longrightarrow
 H^{0}({L^{\otimes (t+1)}}),
\]
{which, according to
Notation~\ref{notation.section.proj.normality} are
\[
H^{0}(L_{n}^{{\otimes ts}})\otimes
H^{0}(L_{n}^{\otimes s})\xrightarrow{\ {\alpha(ts,s;n)}\ }
H^{0}(L_{n}^{{\otimes (t+1)s}}),
\]
surject for all $t \geq 1$}.
{For this, by Lemma~\ref{divideandrule} it is enough to prove
that the multiplication maps
\[
H^{0}(L_{n}^{\otimes t'})\otimes H^{0}(L_{n})\xrightarrow{\ \alpha(t',1;n)\ }
H^{0}(L_{n}^{\otimes (t'+1)})
\]
surject for all $t' \geq n+1$.}
{Recall} that
$L_{n}=K_{\overline X}$. {As observed in Remark~\ref{easyremark}, if $t' \geq n+2$,
then
$\alpha(t',1;n)$ surjects by \cite[p. 41, Theorem 2]{Mumford} and
the Kawamata-Viehweg vanishing theorem. However, Remark~\ref{easyremark}
also says that these standard methods do not yield the surjectivity
of $\alpha(n+1,1;n)$. Therefore the proof of the surjectivity of  the map
$\alpha(n+1,1;n)$}
will be handled in a completely a different way. {We point out (it will be crucial in Step 5) that,
because of our hypothesis, we may assume $p_g(\overline X) > n+1$ when studying $\alpha(n+1,1;n)$.}

\bigskip

\textit{Step 3:} {In order to prove the surjectivity of
$\alpha(n+1,1;n)$
w}e consider, for any $2\leq
n^{\prime}\leq n$, the  commutative diagram

\begin{equation}\label{comm.diagram.restriction}
\begin{matrix}
{\scriptstyle 0} & {\scriptstyle\rightarrow} & {\scriptstyle
 H^{0}(L_{n^{\prime}}^{\otimes n+1}) \otimes H^{0}(\mathcal{O}_{X_{n^{\prime}}})}&
{\scriptstyle \rightarrow} &
{\scriptstyle  H^{0}(L_{n^{\prime
}}^{\otimes n+1}) \otimes H^{0}(L_{n^{\prime}})} &
{\scriptstyle\rightarrow} &
{\scriptstyle  H^{0}%
(L_{n^{\prime}}^{\otimes n+1}) \otimes H^{0}(L_{n^{\prime}-1})}& {\scriptstyle \rightarrow} &
{\scriptstyle 0}\\
&&\downarrow && \downarrow && \downarrow\\
0 &\rightarrow & H^{0}(L_{n^{\prime}}^{\otimes n+1}) & \rightarrow &
H^{0}%
(L_{n^{\prime}}^{\otimes n+2})& \rightarrow &
H^{0}(L_{n^{\prime}-1}^{\otimes
n+2}) &\rightarrow & 0
\end{matrix}
\end{equation}
{Both horizontal sequences in commutative
diagram~\eqref{comm.diagram.restriction} are exact.
Indeed, the
top horizontal row is exact on the right  because $X_{n^{\prime}}$
is regular, as
 shown in Step
1, and the  bottom horizontal row is exact on the right
because $H^{1}(L_{n^{\prime}%
}^{\otimes n+1})=0$ by the Kawamata-Viehweg vanishing theorem.}
{
Note that $\alpha(n+1,1;n')$ is the middle vertical map of
\eqref{comm.diagram.restriction}, so $\alpha(n+1,1;n)$ is
the middle vertical map of
\eqref{comm.diagram.restriction}
when $n'=n$. Then we prove
the surjectivity of $\alpha(n+1,1;n)$ by induction on $n'$.
Thus assume $\alpha(n+1,1;n'-1)$ surjects for $2 \leq n' \leq n$.
The left-hand-side vertical map of \eqref{comm.diagram.restriction}
is an isomorphism. The
right-hand-side vertical map of  \eqref{comm.diagram.restriction}
 is the composition of the map
\begin{equation}\label{map.restriction}
 H^{0}(L_{n^{\prime}}^{\otimes n+1}) \otimes
H^{0}(L_{n^{\prime}-1})\longrightarrow
 H^{0}(L_{n^{\prime}-1}^{\otimes n+1})\otimes
H^{0}(L_{n^{\prime}-1})
\end{equation}
and $\alpha(n+1,1;n'-1)$. The map in \eqref{map.restriction} surjects
because
}
{
$H^{1}(L_{n^{\prime}}^{\otimes
n})=0$, by adjunction and the Kawamata-Viehweg vanishing theorem,
since $n^{\prime}\geq 2$.}
{
The map $\alpha(n+1,1;n'-1)$ surjects by
induction hypothesis, so the right-hand-side vertical map
of  \eqref{comm.diagram.restriction} surjects.
Thus the middle vertical map $\alpha(n+1,1;n')$ surjects.
Therefore the only thing left to check is the first step
of the induction, namely, the surjectivity of
the multiplication map
\[
  H^{0}(L_{1}^{\otimes n+1}) \otimes H^{0}(L_{1})
\xrightarrow{\ \alpha(n+1,1;1)\ } H^{0}(L_{1}^{\otimes n+2})
\]
of global sections on $C$.
For brevity we denote $\beta=\alpha(n+1,1;1)$} and $\theta=L_{1}$.
By adjunction we have $\theta^{\otimes n}%
=K_{C}$.
No conventional results on curves will suffice to
show the surjectivity of $\beta$ due to lack of sufficient positivity
of the
line bundles
$\theta^{\otimes n+1}$ and $\theta$.
So we
have to handle this in a different way. This leads us to the
two, final steps of this proof.

\bigskip

\textit{Step 4:} In the previous steps  we have just
used that $\overline X$ is a
regular variety of
general type. The final steps require
the full force of the fact that $X$ is a Horikawa variety.
{By \cite[Propositions 2.2, 2.5]{Kobayashi}, we know that the
canonical map of $\overline X$ is a morphism, which is
{finite} of  degree $2$ onto a   variety  of minimal
degree.}
Thus the linear system $|\theta|$ induces the
{morphism} $\pi: C\rightarrow D$,
where $\pi$ is
{finite}
of degree $2$. Since $Y$
is a variety of minimal degree,
$D$ is a rational normal curve of degree $r$.
{In addition,
since $D$ is cut-out
by general hyperplane sections, $D$ is smooth.}
Then we have
\begin{equation*}
 \pi_{\ast}\mathcal{O}_{C}=\mathcal{O}_{\mathbb{P}^{1}}
\oplus\mathcal O_{\mathbb{P}^{1}}(-a).
\end{equation*}
 By relative
duality we have $a={nr+2}$, so
\[
\pi_{\ast}\mathcal{O}_{C}=\mathcal{O}_{\mathbb{P}^{1}}\oplus\mathcal{O}%
_{\mathbb{P}^{1}}\mathcal{(-}nr-2).
\]

\bigskip
\textit{Step 5:} Recall that to finish the proof we have
to show that the multiplication map of global sections on $C$
\begin{equation*}
H^{0}(\theta^{\otimes n+1}) \otimes H^{0}(\theta)
\xrightarrow{\ \beta \ } H^{0}(C,\theta^{\otimes n+2})
\end{equation*}
 is
surjective. Due to Step 4 and since $\theta=
\pi^{\ast}(\mathcal{O}_{\mathbb{P}^{1}}(r))$, we have
\begin{equation*}
 H^{0}(\theta)=H^{0}(\pi_{\ast}\mathcal{\theta})
=H^{0}(\mathcal{O}_{\mathbb{P}^{1}}(r))\oplus
H^{0}(\mathcal{O}_{\mathbb{P}^{1}}\mathcal{((}1\mathcal{-}n)r-2)),
\end{equation*}
 hence
$H^{0}(\theta)=H^{0}(\mathcal{O}_{\mathbb{P}^{1}}(r))$. We also
have
\begin{equation*}
 H^{0}
(\theta^{\otimes n+1})=H^{0}(\pi_{\ast}(\mathcal{\theta}^{\otimes n+1})
)=H^{0}(\mathcal{O}_{\mathbb{P}^{1}}((n+1)r))\oplus H^{0}(\mathcal{O}
_{\mathbb{P}^{1}}\mathcal{(}r-2)).
\end{equation*}
Note that
\begin{equation*}
 \pi_{\ast}(\mathcal{O}_{C})=\mathcal{O}_{\mathbb{P}^{1}}\oplus
\mathcal{O}_{\mathbb{P}^{1}}\mathcal{(-}nr-2)
\end{equation*} is a sheaf of $\mathcal{O}%
_{\mathbb{P}^{1}}$-algebras whose ring multiplication decomposes in the
following way: The map
\[
\mathcal{O}_{\mathbb{P}^{1}}\otimes\mathcal{O}_{\mathbb{P}^{1}}\rightarrow
\mathcal{O}_{\mathbb{P}^{1}}%
\]
is the ring multiplication in $\mathcal{O}_{\mathbb{P}^{1}}$. The maps%
\begin{align*}
\mathcal{O}_{\mathbb{P}^{1}}\otimes\mathcal{O}_{\mathbb{P}^{1}}\mathcal{(-}%
nr-2)  & \rightarrow\mathcal{O}_{\mathbb{P}^{1}}\mathcal{(-}nr-2)\\
\mathcal{O}_{\mathbb{P}^{1}}\mathcal{(-}nr-2)\otimes\mathcal{O}_{\mathbb{P}%
^{1}}  & \rightarrow\mathcal{O}_{\mathbb{P}^{1}}\mathcal{(-}nr-2)
\end{align*}
are the left and right module multiplication of $\mathcal{O}_{\mathbb{P}^{1}%
}\mathcal{(-}nr-2).$ The map
\[
\mathcal{O}_{\mathbb{P}^{1}}(-nr-2)\otimes\mathcal{O}_{\mathbb{P}^{1}%
}(-nr-2)\rightarrow\mathcal{O}_{\mathbb{P}^{1}}%
\]
is given by structure of double cover of $\pi$.
 Denote
\begin{align*}
A(m)  & =H^{0}(\mathcal{O}_{\mathbb{P}^{1}}(mr)),\\
A'(m)  & =H^{0}(\mathcal{O}_{\mathbb{P}^{1}}\mathcal{((}m-n)r-2)).
\end{align*}
In this notation
\begin{align*}
H^{0}(\theta^{\otimes n+1}) = A(n+1)\oplus A'(n+1),  \\
H^{0}(\theta)= A(1)\oplus A'(1) = A(1),
\end{align*}
 since $A'(1)=0$, and
 \begin{equation*}
H^{0}(\theta
^{\otimes n+2})=A(n+2)\oplus A'(n+2).
 \end{equation*}
Therefore the map $\beta$ splits as direct
sum of the maps%
\begin{align*}
& A(n+1)\otimes A(1)\overset{\beta_{1}}{\longrightarrow}A(n+2)\\
& A'(n+1)\otimes A(1)\overset{\beta_{2}}{\longrightarrow}A'(n+2).
\end{align*}
The map $\beta$ surjects if
{and only if}
the maps $\beta_{1}$ and $\beta_{2}$ both surject.
{On $\mathbb P^1$ the multiplication map of global sections
\begin{equation*}
 H^0(\mathcal O_{\mathbb P^1}(c_1)) \otimes H^0(\mathcal O_{\mathbb P^1}(c_2))
 \longrightarrow H^0(\mathcal O_{\mathbb P^1}(c_1+c_2))
\end{equation*}
surjects if and only if $c_1$ and $c_2$ are both nonnegative.}
{Recall that $A(1)=H^0(\mathcal O_{\mathbb P^1}(r))$ and $A(n+1)=H^0(\mathcal O_{\mathbb P^1}((n+1)r)$
and $r, (n+1)r >0$. On the other hand, $A'(n+1)=H^0(\mathcal O_{\mathbb P^1}(r-2))$.
Under our hypothesis, $p_g > n+1$, so the degree of $D$ is $r=p_{g}(X)-n \geq 2$ and $r-2 \geq 0$.
Then $\beta$ surjects as wished.}
\end{proof}

 \section{Hyperelliptic subcanonical polarized  varieties}

Even though this article focuses on
 Horikawa varieties, the
 main results of Sections 2 and 3, when considered for strong
 Horikawa varities, hold more generally.
 Indeed, if $X$ is a strong Horikawa variety, the
 polarized variety $(X, \omega_X)$ is hyperelliptic
 according to the following
 definition (compare with \cite[Definition 1.1]{Fujita}):

\begin{definition}
{\rm Let $(X,A)$ be a polarized variety.
If $A$ is base-point-free,
the morphism induced by $|A|$ has degree $2$ and its
image is a variety of minimal degree, then
$(X,A)$ is a hyperelliptic polarized variety.}
 \end{definition}

 We also recall the definition of subcanonical variety
 (compare with \cite[Definition 2.11]{BGMR}):

 \begin{definition}
{\rm Let $(X,A)$ be a polarized variety,
let $A$ be base-point-free. If $\omega_X=A^{\otimes s}$, then
we say $(X,A)$ is $s$--subcanonical.}
\end{definition}

\begin{prop}\label{prop.class.hyperelliptic}
 {Let $(X,A)$ be an $s$--subcanonical
 polarized variety of dimension $n$.
 If $(X,A)$ is hyperelliptic, then
 $s \geq -n+1$.
 }
\end{prop}

\begin{proof}
We call $\varphi$ to the morphism induced by $|A|$ and
argue as in the proof of Theorem~\ref{regular} and
construct $\overline X_2$, $\phi_2$ and $Y_2$ in the same fashion.
If $Y_2$ is $\PPP^2$, then the branch divisor of $\phi_2$ is
$2(n+s+1)$ times a line in $\PPP^2$. Since the branch divisor
is effective and $\overline X_2$ is connected,
then $n+s \geq 0$. If $n+s=0$, then $\varphi$ is not induced
by the complete linear series $|A|$ of $A$, hence $n+s \geq 1$ in
this case.

\smallskip
If $Y_2$ is the Veronese surface, then
the branch divisor of $\phi_2$ is
$2(2n+2s-1)$ times a line in $\PPP^2$.
Since the branch divisor
is effective and $\overline X_2$ is connected,
$n+s \geq 1$ in this case.

\smallskip
If $Y_2$ is smooth rational normal scroll or a cone over a
rational normal curve (in the latter case, we carry out
a construction analogous to \eqref{CD}),
then we get results analogous to \eqref{ramification},
\eqref{eq.for.boundedness1}, \eqref{formula.branch.singular.case}
and \eqref{eq.for.boundedness2}. Thus
the branch divisor of $\phi_2$ is linearly equivalent
to
\begin{equation}\label{eq.ramification.polarized}
 2(n+s)C_0+2((n+s-2)m+e+2)f,
\end{equation}
with $m \geq e+1$ and
the branch divisor of $p_2$ is linearly equivalent
to
\begin{equation*}
  2(n+s)C_0+2(n+s-e+2)f.
\end{equation*}
Since these branch divisors are effective, we get
$n +s \geq 0$. If $n+s=0$ and $Y_2$ is smooth,
since $\overline X_2$ is connected, $2m \leq e+1$,
but this contradicts $m \geq e+1$.
If $n+s=0$ and $Y_2$ is singular,
since $\overline X_2$ is connected, $e \leq 1$, which again
is a contradiction.
Thus $n+s \geq 1$ in
both cases.
\end{proof}

Now we are ready to state the analogues of the main results of Sections 2 and 3.
The analogue of Theorem~\ref{topology}
generalizes what \cite[Corollary 5.17]{Fujita}  says with respect to the simple connectedness of hyperelliptic, subcanonical polarized varieties (see the comment about this
just before Theorem~\ref{topology}). The analogues of Theorems~\ref{theorem.p_g.bounded} and
 \ref{regularity} (1) are novel though.

\begin{thm}\label{thm.analogue}
 If in the statements  of
 Theorems~\ref{regular}, \ref{theorem.p_g.bounded},
 \ref{regularity} (1) and \ref{topology} we replace ``$X$ a
 Horikawa variety'' by
 ``\,$(X,A)$ a hyperelliptic, subcanonical polarized variety
 with
 canonical  singularities'',
 ``\,the canonical morphism of $X$'' by
 ``\,the morphism induced by $|A|$'',
 and ``\,$p_g(X)$''
by ``\,$h^0(A)$'', then the same conclusions of those
theorems hold in this new setting.

\smallskip
In particular, the above statements hold for
$(X,A)$ polarized, hyperelliptic Calabi--Yau varieties
 with
 canonical  singularities and for $(X,A)$
 polarized, hyperelliptic
 Fano varieties of index $i \leq n-3$ (in the sense of
 Fujita, see \cite[Definition 1.5]{Fujita}),
 with canonical singularities.
\end{thm}

\begin{proof}
 In the proofs of Theorems~\ref{regular}, \ref{theorem.p_g.bounded},
and \ref{topology} we showed that  $\overline X_2$ is regular,
that $p_g(\overline X_2) \leq 6$
 if $Y_2$ and that $\overline X_2$ is simply connected.
 If we set $s'=n-1$,
 in all those cases, we in fact proved those statements for
 any hyperelliptic, $s'$--subcanonical polarized
 normal surface
 $(\overline X_2, \phi_2^*(\mathcal O_Y(1)))$
 with canonical singularities and $s' \geq 1$.
 Actually, it is easy to check that the same arguments go through and the
 same statements are true
 for $s'=0$ and for $s'=-1$, with the possible exception of
 Theorem~\ref{topology} in the latter case.
 Thus, if $(X,A)$ is a
 hyperelliptic, $s$--subcanonical polarized
 with canonical singularities
 and $s \geq -n+1$, then we define
 $\overline X_2$ from $(X,A)$ in a  way
 anologuos to the proofs of
Theorems~\ref{regular}, \ref{theorem.p_g.bounded},
and \ref{topology}, so we get analogous conclusions except maybe
for the above mentioned exception.

\smallskip

 We deal now with the exception. In that case, the arguments
 of the proof of Theorem~\ref{topology} do work except
 maybe if $Y_2=\mathbb F_0$. Then (see \eqref{eq.ramification.polarized})
 \begin{equation*}
  B \sim 2C_0 +(4-2m)f.
 \end{equation*}
 Since $B$ is effective, then $m=1$, in which case $B$ is ample, or
 $m=2$. Since $\overline X_2$ is normal, in the latter case $B$ is the
 union of two lines and $\overline X_2=\mathbb F_0$, so
 $\overline X_2$ is simply connected in both cases and, arguing as in the
 proof of Theorem~\ref{topology}, $X$ is also simply connected.

\smallskip
We deal now with the analogue of Theorem~\ref{regularity}
(1). If $Y$ is $\PPP^n$ or a smooth hyperquadric in $\PPP^{n+1}$,
the same arguments of the proof of Theorem~\ref{regularity} work.
Now let $Y$ be a smooth rational normal scrol. Using the same notation
as there, the analogue
of \eqref{equation.regularity} yields
\begin{equation*}
 \mathcal O_Y(K_Y) \otimes \mathcal L= \mathcal O_Y(s),
\end{equation*}
so
\begin{equation*}
 h^i(\mathcal L^{-1})=h^{n-i}(\mathcal O_Y(s)).
\end{equation*}
If $s <0$, then $h^{n-i}(\mathcal O_Y(s))=0$ by Serre duality and the
Kodaira vanishing theorem. If  $s \geq 0$, then
\begin{equation*}
h^{n-i}(\mathcal O_Y(s))=h^{n-i}(S^s(\mathcal E))=0.
\end{equation*}
\end{proof}

\begin{rem}
 {\rm Because of the analogue of Theorem~\ref{regular},
 hyperelliptic, $0$--subcanonical polarized varieties
 with canonical singularities are Calabi--Yau.}
\end{rem}

\end{document}